\documentclass{article}

\usepackage{amssymb}
\usepackage{amsmath}
\usepackage{graphicx}
\newtheorem{theorem}{Theorem}
\pdfoutput=1

\newtheorem{corollary}[theorem]{Corollary}

\newtheorem{definition}[theorem]{Definition}

\newtheorem{lemma}[theorem]{Lemma}

\newtheorem{proposition}[theorem]{Proposition}

\newenvironment{proof}[1][Proof]{\noindent\textbf{#1.} }{\ \rule{0.5em}{0.5em}}

\begin{document}

\title{On the boundary of closed convex sets in $E^{n}$}
\maketitle

\centerline{\bf {M. Beltagy}}

\centerline{Faculty of Science, Tanta University, Tanta, Egypt}
  
\centerline{\em{E-mail: \tt beltagy50@yahoo.com.}}

\centerline{\bf {S. Shenawy}}

\centerline{Modern Academy for Engineering and Technology in Maadi,} 
\centerline{El-Hadaba El-Wosta, Mokatem Close to Te. Ph. Exchange, Maadi, Egypt.}
\centerline{\em E-mail address: \tt drssshenawy@eng.modern-academy.edu.eg}

\begin{abstract}
In this article a class of closed convex sets in the Euclidean $n$-space which are the convex hull of their profiles is described. Thus a generalization of Krein-Milman theorem\cite{Lay:1982} to a class of closed non-compact convex sets is obtained. Sufficient and necessary conditions for convexity, affinity and starshapedness of a closed set and its boundary have been derived in terms of their boundary points.
\end{abstract}

M.S.C 2000: 53C42, 52A20;


\section{An introduction}
	It is very nice to study the boundary of a set $A$ in the Euclidean $n-$space $E^{n}$ and get a global properties of $A$. In the first part of this article, we get sufficient conditions for the convexity, affinity, and starshapedness of $A$ and its boundary in terms of boundary points. Some relations between geometrical and topological properties were derived. The Krein-Milman theorem \cite{Lay:1982} in the $n$-dimensional Euclidean space $E^{n}$ asserts that every compact convex set is the convex hull of its extreme points i.e. given a compact convex set $A\subset E^{n}$, one only need its extreme points $E\left( A\right) $ to recover the set shape. Moreover, if $T$ is a subset of $A$ and the convex hull of $T$ is $A$, then $E\left( A\right) \subset T$. Therefore, $E\left(A\right) $ is the the smallest set of points of $A$ for which $A$ is the convex hull. The compactness condition in this theorem is important. For example, consider the set $A=\left\{ \left( x,y\right) :y\geq \left\vert x\right\vert \right\} $. $A$ is a closed convex subset of $E^{2}$ and has one extreme point namely $\left( 0,0\right) $ i.e. $A$ is not equal to the convex hull of its extreme points. On the other side, we notice that there are many non-compact convex sets which are equal to the convex hull of their extreme points so this theorem is not a characterization of compact convex sets. In \cite{Li:2006} it is proved that there are some planar non-compact convex sets whose minimal convex generating subset does exist. In this article the authors generalize Krein-Milman theorem to the setting of closed convex sets in $E^{n}$ extending the previous work on compact convex sets. The reader may refer to \cite{Balashov:2002, Li:2006, GARCIA:1985, Oates:1971, Lay:1982, Valentine:1964} for more discussions about Krein-Milman theorem in different spaces.

We begin with some definitions from  \cite{MarilynB:2005, Lay:1982, Valentine:1964}. Let $A$ be a subset of the Euclidean $n-$space $E^{n}$. A point $x\in A$ sees $y\in A$ via $A$ if the closed segment $ \left[ xy\right] $ joining $x$ and $y$ is contained in $A$. Set $A$ is called starshaped if and only if for some point $x$ of $A$, $x$ sees each point of $A$ via $A$ and the set of all such points $x$ is called the (convex) kernel of $A$ and is denoted by $\ker A$. $A$ is convex if and only if $A=\ker A$. A set $A$ is said to be affine if $x,y\in A$ implies that the line passing through $x,y$ lies in $A$. The convex hull $C\left( A\right) $ of $A$ is the intersection of all convex subsets of $E^{n}$ that contain $A,$ where the closed convex hull of $A$ is defined as the intersection of all closed convex subsets of $E^{n}$ that contain $A$. Let $A$ be a convex set. A point $x$ in $A$ is called an extreme point of $A$ if there exists no non-degenerate line segment in $A$ that contains $x$ in its relative interior. The set of all extreme points of $A$ is called the profile of $A$ and is denoted by $E(A)$. 

Throughout this paper, $intA,$ $\bar{A},$ $A^{c},$ and $\partial A$ will denote the interior, closure, complement, and boundary of $A$. For two
different points $x$ and $y$, $\overrightarrow{xy}$ will represent the ray emanating from $x$ through $y$, and $\overleftrightarrow{xy}$ will be the
corresponding line.

\section{Results}

We begin this section by the following important definitions.

\begin{definition}
Let $A$ be any non-empty subset of $E^{n}$. Let $p,q$ be two distinct boundary points of $A$. The pair $p$ and $q$ is called parabolic(flat) if $%
\left( pq\right) $ is contained in $\partial A$, hyperbolic if $\left(pq\right) $ is contained in $int\left( A\right) $,and elliptic if $\left(
pq\right) $ is contained in $A^{c}$ where $\left(pq\right)=\left[pq\right]\setminus \left\{p,q\right\}$ is the open segment joinning $p$ and $q$.
\end{definition}

It is easy to find pair of boundary points which is not one of the three
types. Now we begin our results by the following theorem.

\begin{theorem}
\label{th a}Let $A$ be a non-empty closed subset of $E^{n}$. If every pair
of boundary points of $A$ is flat, then $A$and $\partial A$ are both convex
sets. Moreover, if $A$ has a non-empty interior, then $A$ is unbounded, $%
\partial A$ is affine and $A^{c}$ is convex.
\end{theorem}

\begin{proof}
The convexity of $\partial A$ is direct and so we consider the convexity of $ A$ only. Suppose that $A$ is not convex i.e. we get two points $p$ and $q$ in $A$ such that $\left( pq\right) \nsubseteq A$. Since $A$ is closed, we find two points $r,s$ in $\partial A$ such that $\left( rs\right) \cap
A=\phi $ which is a contradiction and so $A$ is convex.

To show that $A$ is unbounded, let $p\in int\left( A\right) $. Suppose that $%
A$ is bounded and so we find a real number $\varepsilon $ such that $A$ is
contained in the closed ball $\overline{B}\left( p,\varepsilon \right) $ of
radius $\varepsilon $ and center at $p$. Let $\left[ ab\right] $ be any
chord of $\overline{B}\left( p,\varepsilon \right) $ that runs through $p$.
Since $A$ and $\left[ ab\right] $ are both closed and convex sets, we find $%
a^{\prime }$ and $b^{\prime }$ in $\partial A$ such that $A\cap \left[ ab%
\right] =\left[ a^{\prime }b^{\prime }\right] $ which is a contradiction
since $\left[ a^{\prime }b^{\prime }\right] $ cuts the interior of $A$. Therefore $A$ is unbounded.

We now consider the affinity of $\partial A.$ Suppose that $\partial A$ is
not affine i.e. there are two points $a$ and $b$ in $\partial A$ such that
the line $\overleftrightarrow{ab}$ passing through $a$ and $b$ is not
contained in $\partial A.$ Since $\partial A$ and $\overleftrightarrow{ab}$
are closed convex sets, there are $a^{\prime }$ and $b^{\prime }$ in $%
\partial A$ such that 
\[
\left[ ab\right] \subset \partial A\cap \overleftrightarrow{ab}=\left[
a^{\prime }b^{\prime }\right] 
\]

Let $p\in \overleftrightarrow{ab}\backslash \left[ a^{\prime }b^{\prime }\right] $ (i.e. $p\in int\left( A\right) \cap \overleftrightarrow{ab}$ or $%
p\in A^{c}\cap \overleftrightarrow{ab}$) and assume that $p\in \overrightarrow{b^{\prime }a^{\prime }}$. If $p\in int(A)\cap \overleftrightarrow{ab}$, then the convex cone $C\left( b,\overline{B}\left( p,\varepsilon \right) \right) $ with vertex $b$ and base $\overline{B}\left( p,\varepsilon \right) $ for a sufficiently small $\varepsilon $ shows \ that $a$ is an interior point which is a contradiction see Figure \ref{07-01}.

\begin{figure}[h]
\begin{center}
\includegraphics{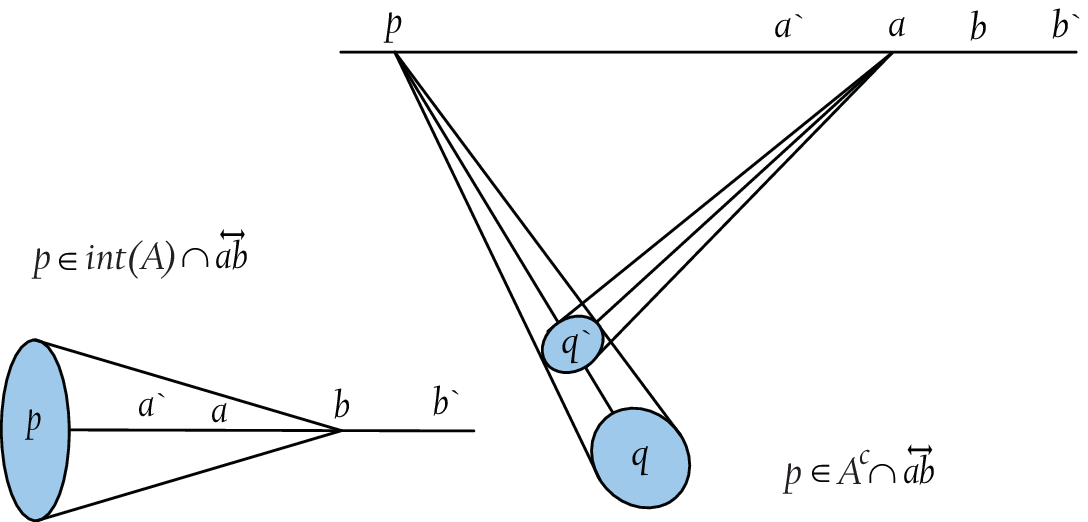}
\caption{Two cases for the point $p$} 
\label{07-01}
\end{center}
\end{figure}

Now we take $p\in A^{c}\cap \overleftrightarrow{ab}$. Let $q$ be a point of $int\left( A\right)$. The sets $\left[ pq\right] $ and $A$ are closed convex sets and so there is a point $q^{\prime }\in \partial A$ such that $\left[ pq\right] \cap A=\left[ q^{\prime }q\right] $. This implies that the
intersection $B=\partial A\cap C\left( p,\overline{B}\left( q,\epsilon \right) \right) $, for a small $\epsilon $, is a non-empty closed convex
set since $\partial A$ is convex. Therefore, $B$ is a convex cross section of $C\left( p,\overline{B}\left( q,\epsilon \right) \right)$ that determines a hyperplane $H$ whose intersection with $C\left( p,\overline{B}\left( q,\epsilon \right) \right)$ is $B$. At least one of the points $a$ and $b$(say $a$) does not lie in $H$ otherwise the line $\overleftrightarrow{ab}$ lies in $H$ which contradicts the fact that $p$ is the vertex of the convex cone $C\left( p,\overline{B}\left( q,\epsilon \right) \right) $. Now, the convex cone $C\left( a,B\right) $ has dimension $n$ i.e. $C\left( a,B\right) $ has interior points which is a contradiction since both $a$ and $B$ are in $\partial A$ see Figure \ref{07-01}. This contradiction completes the proof.
\end{proof}

\begin{corollary}
Let $A$ be a non-empty open subset of $E^{n}$ and $int(\overline{A})=A$. If
every pair of boundary points of $A$ is flat, then $A$ is unbounded convex
set and $\partial A$ is affine.
\end{corollary}

\begin{proof}
It is clear that $\overline{A}$ satisfies the hypothesis of Theorem \ref{th
a}. Therefore $\partial \overline{A}=\partial A$ is affine and $\overline{A}$
is convex and unbounded. Note that if $A$ is bounded, then $\overline{A}$ is
also bounded and equivalently, $\overline{A}$ is unbounded implies that $A$
is unbounded. Since the interior of a closed convex set is also convex, the
convexity of $\overline{A}$ implies that $A$ is convex.
\end{proof}

\begin{theorem}
Let $A$ be a non-empty closed subset of $E^{n}$. If every pair of boundary
points of $A$ is hyperbolic, then $A$ is strictly convex.
\end{theorem}

\begin{proof}
It is enough to prove that $A$ is convex since the strict convexity of $A$
is direct. Now, we assume that $A$ is not convex i.e. there are $p,q$ in $A$
such that $\left[ pq\right] $ is not contained in $A.$ Since $A$ is closed,
there are $p^{\prime },q^{\prime }$ in $\partial A$ such that $\left(
p^{\prime }q^{\prime }\right) \cap A=\phi $ which is a contradiction and $A$
is convex.
\end{proof}

\begin{corollary}
Let $A$ be a non-empty closed subset of $E^{n}$. $A$ is convex if and only
if each pair of boundary points is either flat or hyperbolic.
\end{corollary}

Since the interior of a closed convex set is again convex, this result is still true for open sets such that $int\left( \overline{A}\right) =A$ i.e.
if an open set $A$ satisfies $int\left( \overline{A}\right) =A$  and every pair of $\partial A=\partial \overline{A}$ is either flat or hyperbolic then its closure $\overline{A}$ satisfies the hypotheses of the above corollary and hence $\overline{A}$ and $A$ are both convex sets.
The following example shows that the closeness is important. Let $A$ be a subset of $E^{2}$ defined by $A=\left\{ \left( x,y\right) :0\prec x\prec 1,0\prec y\prec 1\right\} \cup \left\{ \left( 0,0\right) ,\left( 1,1\right) ,\left(1,0\right) ,\left( 0,1\right) \right\} .$ $A$ is neither closed nor open and every pair of boundary points is either flat or hyperbolic but $A$ is not convex.

\begin{corollary}
Let $A$ be a non-empty open subset of $E^{n}$ and $int(\overline{A})=A$. If every pair of boundary points of $A$ is elliptic, then $A^{c}$ is strictly convex. Also, if every pair of boundary points of $A$ is either elliptic or parabolic, then $A^{c}$ is convex.
\end{corollary}

\begin{proposition}
Let $A$ be a non-empty closed subset of $E^{n}.$ If there is a point $p\in A$ that sees $\partial A$ via $A$, then $A$ is starshaped.
\end{proposition}

\begin{proof}
We claim that $p\in \ker A$. Suppose that $p$ is not in $\ker A$ i.e. there
is a point $q\in A$ such that $\left[ pq\right] $ is not contained in $A$. Since $A$ is closed, there are two points $p^{\prime}$ and $q^{\prime}$ in $\partial A\cap \left[ pq\right] $ such that $\left( p^{\prime}q^{\prime}\right)\cap A=\phi$. Thus $p$ does not see neither $p^{\prime}$ nor $q^{\prime}$. This contradicts the fact that $p$ sees $\partial A$ via $A$ and the proof is complete.
\end{proof}

It is clear that the converse of this result is also true. Thus we can say that this proposition is a characterization of the kernel of the closed starshaped sets. This means that the kernel of a closed starshaped set $A$ is only the points of $A$ that see $\partial A$. The
following corollary is direct.

\begin{corollary}
Let $A$ be a non-empty closed subset of $E^{n}.$ If $\partial A$ is starshaped, then $\ker \left( \partial A\right) \subset \ker A.$
\end{corollary}

In the light of the above results, one can test the convexity and starshapedness of a closed set $A$ using its boundary points. In the next part a minimal subset of these boundary points will build $A$ up from inside.

\begin{theorem}
\label{th b} Let $A$ be a non-empty closed convex subset of $E^{n}$. If $A$ has no hyperplane, then $A=C\left( \partial A\right)$.
\end{theorem}

\begin{proof}
Since $A$ is convex, $A$ is connected. We will prove that $C\left( \partial A\right) $ is open and closed in the relative topology on $A$ and hence $A=C\left( \partial A\right)$.

First, we prove that $C\left( \partial A\right) $ is open in $A$. Let $p\in C\left( \partial A\right) \subset A$. We have the following cases:

\begin{enumerate}
\item $p\in C\left( \partial A\right) \cap int\left( A\right) $: Let $B_{\delta}=B\left(p,\delta \right)\cap A$. In this case there exists a real number $\delta $ such that $B\left( p,\delta \right) \subset A$ and so $B_{\delta}=B\left(p,\delta \right)$. Suppose that $p$ is not an interior point of $C(\partial A)$ i.e. for any $\delta $, $B_{\delta}$ is not contained in $C\left( \partial A\right)$ and so $p$ is a boundary point of $C(\partial A)$. Therefore, there is a supporting hyperplane $H_{1}$ of $\overline{C(\partial A)}$(the closure of $C(\partial A)$ is a closed convex subset of $A$) at $p$ and $\overline{C(\partial A)}$ is contained in a closed half-space with boundary $H_{1}$. Let $x$ be any point of $B\left( p,\delta \right) $ that lies on the other side of $H_{1}$ and let $H_{2}$ be a parallel hyperplane to $H_{1}$ at $x$. Since $A$ does not contain a hyperplane, we find a point $y\in H_{2}\setminus A$. The line segment $\left[ xy\right] $ cuts $\partial A$ at a point $z\in H_{2}$ which contradicts the fact that $H_{1}$ supports $\overline{C\left(\partial A\right)}$. This contradiction implies that $p$ is an interior point of $C\left( \partial A\right) $ in the relative topology of $A$.
\item $p\in C\left( \partial A\right) \cap \partial A$: in this case, $B\left( p,\delta \right) $ has a non-empty intersection with $A$ for any real number $\delta $. Let $B_{\delta }=B\left( p,\delta \right)\cap A$. Suppose that $p$ is not an interior point of $C\left( \partial A\right)$. Then, for any $\delta$, the set $B_{\delta }$ has a point $x$ which is not in $C\left( \partial A\right)$. But $\overline{C\left( \partial A\right)} $ is closed convex set and $x\notin \overline{C\left( \partial A\right)}$, and so we get a hyperplane $H$ passing through $x$ that separates $x$ and $\overline{C\left( \partial A\right)}$. Since $A$ does not have a hyperplane, there is a point $y$ in $H\setminus A$ where $[xy]$ cuts $\partial A$. Thus $H$ cuts $\partial A$ and so $H$ cuts $C\left( \partial A\right) $ which is a contradiction and so $p$ is an interior point of $C\left( \partial A\right) $ in the relative topology on $A$.
\end{enumerate}

This discussion above implies that $C\left( \partial A\right) $ is an open set in $A$.
Now, we want to prove that $C\left( \partial A\right) $ is closed in $A$. Let $p$ be a boundary point of $C\left( \partial A\right)$.
If $p\in \partial A,$ then $p\in C\left( \partial A\right) .$ Let $p\in intA $, then there is a small positive real number $\delta $ such that $B\left( p,\delta \right) \subset A.$ Since $p$ is a boundary point of $C\left( \partial A\right) $, $B\left( p,\delta \right) \neq B\left( p,\delta
\right) \cap C\left( \partial A\right) \neq \phi $. Therefore, we find a point $x$ in $B\left( p,\delta \right) $ which is not in $C\left( \partial
A\right) $. Since $\overline{C\left( \partial A\right)} $ is a closed convex set, we get a hyperplane $H$ passing through $x$ and does not intersect $C\left( \partial A\right) $. But $A$ does not have a hyperplane and so $H$ cuts $\partial A$ which is a contradiction and $p\in C\left( \partial A\right) $ i.e. $C\left( \partial A\right) $ is closed in the relative topology on $A$ and the proof is complete.
\end{proof}

In general, sets need not have extreme points. The following proposition gives a sufficient condition for the existence of extreme points.

\begin{proposition}
Let $A$ be a non-empty closed convex subset of $E^{n}$. $A$ contains at least one extreme point if and only if $A$ has no line.
\end{proposition}

\begin{proof}
Let us assume that $A$ has a line $l$. Suppose that $A$ has an extreme point $p$. It is clear that $p\notin l$. Let $B$ be the closed convex hull of $p$ and $l$. $B$ is a subset of $A$ since $A $ is a closed convex set containing both $p$ and $l$. It is clear that $B$ contains a line passing through $p$ and parallel to $l$ i.e. either $p$ is not an extreme point or the line $l$ does not exist.
\end{proof}

\begin{lemma}
Let $A$ be a non-empty closed convex subset of $E^{n}$. If $H$ is a supporting hyperplane of $A$, then $E\left( H\cap A\right) \subset E\left( A\right) $
\end{lemma}

\begin{proof}
Let $p$ be an extreme point of $H\cap A$. Suppose that $p\notin E\left( A\right) $ i.e. there are $x,y$ in $\partial A$ such that $p\in \left(
xy\right) .$ The hyperplane $H$ supports $A$ at $p$ and so $\left[ xy\right]\subset H$. This implies that $p\in \left[ xy\right] \subset H\cap A$ which contradicts the fact that $p$ is an extreme point of $H\cap A$. This contradiction completes the proof.
\end{proof}

The minimal subset of a compact convex set $A$ which generates $A$ is is its extreme points. Our next main theorem shows that this property is more general.

\begin{theorem}
\label{Th c}Let $A$ be a non-empty closed convex subset of $E^{n}$. If $A$ has no hyperplane and its boundary has no ray, then $A=C\left( E\left(
A\right) \right) $.
\end{theorem}

\begin{proof}
To prove that $A=C\left( E\left( A\right) \right) $, it suffices to prove that $\partial A\subset C\left( E\left( A\right) \right) $ and by Theorem \ref{th b}, we get that $A=C\left( \partial A\right) \subset C\left( E\left( A\right) \right) \subset A$ and hence $A=C\left( E\left( A\right) \right) $.

Let $p\in \partial A$. If $p$ is an extreme point, then $p\in E\left(A\right) \subset C\left( E\left( A\right) \right) $. Now suppose that $p$ is
not an extreme point i.e. there are $x,y$ in $\partial A$ such that $p\in \left( xy\right) $. Since $A$ is a closed convex set, there is a supporting hyperplane $H$ of $A$ at $p$. It is clear that the set $H\cap A$ is a non-empty closed convex subset of $\partial A$. Since $\partial A$ has no ray, the set $H\cap A$ is bounded i.e. $H\cap A$ is a compact convex set. Therefore, $H\cap A=C\left( E\left( H\cap A\right) \right) $. But $E\left( H\cap A\right)
\subset E\left( A\right) $ and so $p\in C\left( E\left( H\cap A\right)\right) \subset C\left( E\left( A\right) \right) $ and the proof is complete.
\end{proof}

\section{Conclusion}

In this work three types of pairs of boundary points were introduced. They are the main tool in the proof of convexity, starshapedness or affinity of closed sets and their boundaries in $E^{n}$. Moreover, we got some relations between topological and geometrical properties of these sets.
Given a convex set $A$, is it possible to find a minimal subset of $A$ whose convex hull is equal to $A$? If $A$ is compact, Krein-Milman theorem says
that the minimal such subset of $A$ is its extremes. We extend this result for a class of closed convex sets. Theorem \ref{Th c} shows that the closed
convex set can be built up from inside by computing the convex hull of its extremes under certain conditions i.e. $A=C\left( E\left( A\right) \right) $. This theorem is an important refinement of Krein-Milman theorem. The set $A=\left\{\left(x,y,z \right): x^{2}+y^{2}=4,-1\prec x\prec 1, 0 \leq z \leq 2 \right\}$ is a non-closed convex subset of $E^{3}$ satisfying $A=C\left( E\left( A\right) \right) $. So the question now is: Does a characterization of all sets satisfying the relation $A=C\left( E\left( A\right) \right) $ in $E^{n}$ exist? The answer of this question is left as an open problem.

\end{document}